\documentclass[10pt,a4paper]{amsart}
\usepackage[english]{babel}
\usepackage{amsmath,amsthm,mathrsfs,hyperref}
\usepackage{amssymb}
\usepackage{aliascnt}
\usepackage{charter}
\usepackage[utf8]{inputenc}
\usepackage[T1]{fontenc}
\usepackage{comment}
\usepackage{todonotes}
\usepackage{xspace}
\usepackage{nicefrac}

\DeclareMathOperator{\acc}{acc}
\DeclareMathOperator{\dom}{dom}

\DeclareMathOperator{\otp}{otp}

\DeclareMathOperator{\cf}{cf}

\DeclareMathOperator{\TP}{TP}
\DeclareMathOperator{\ITP}{ITP}

\newcommand{\ZFC}{{\rm ZFC}\xspace}

\newcommand{\PCF}{{\rm PCF}\xspace}

\newcommand{\bS}{\mathbb{S}}
\newcommand{\bT}{\mathbb{T}}

\newtheorem{theorem}{Theorem}
\newaliascnt{example}{theorem}

\aliascntresetthe{example}
\newaliascnt{fact}{theorem}

\aliascntresetthe{fact}
\newaliascnt{corollary}{theorem}

\aliascntresetthe{corollary}
\newaliascnt{lemma}{theorem}
\newtheorem{lemma}[lemma]{Lemma}
\aliascntresetthe{lemma}
\newaliascnt{claim}{theorem}
\newtheorem{claim}[lemma]{Claim}
\aliascntresetthe{claim}

\theoremstyle{definition}
\newaliascnt{definition}{theorem}
\newtheorem{definition}[definition]{Definition}
\aliascntresetthe{definition}

\newtheorem*{theorem*}{Theorem}

\newtheorem*{example*}{Example}

\newtheorem*{question*}{Question}
\begin{document}
\title{The strong tree property and weak square}
\author{Yair Hayut}
\author{Spencer Unger}
\thanks{The main ideas for this paper were conceived at the Workshop on High and
Low forcing at the American Institute of Mathematics in January 2016.}

\begin{abstract}
We show that it is consistent, relative to $\omega$ many supercompact cardinals, that the super tree property holds at $\aleph_n$ for all $2 \leq n < \omega$ but there are weak square and a very good scale at $\aleph_{\omega}$. 
\end{abstract}
\maketitle

\section{introduction}

In this paper we consider the strong and super tree properties which
characterize strong and super compactness at inaccessible cardinals.  We show
that certain consequences of supercompactness do not follow from the super tree
property. We begin with some definitions. Let $\kappa \leq \lambda$ be
cardinals with $\kappa$ regular.

\begin{definition}
We define a $(\kappa,\lambda)$-tree to be a sequence $T = \langle T_x \mid x \in
P_\kappa\lambda \rangle$ such that for all $x \in P_\kappa\lambda$:
\begin{enumerate}
\item $T_x$ is a nonempty set of functions from $x$ to $2$ and 
\item for all $y \subseteq x$ and all $f \in T_x$, $f \upharpoonright y \in
T_y$.
\end{enumerate}
\end{definition}

\begin{definition} A $(\kappa,\lambda)$-tree $T$ is \emph{thin} if for all $x
\in P_\kappa\lambda$, $\vert T_x \vert < \kappa$. \end{definition}

\begin{definition} A function $b: \lambda \to 2$ is a \emph{cofinal branch}
through a $(\kappa,\lambda)$-tree $T$ if for all $x \in P_\kappa\lambda$, $b
\upharpoonright x \in T_x$. \end{definition}

\begin{definition} We define two reflection properties:
\begin{enumerate}
\item $\TP(\kappa,\lambda)$ holds if every thin $(\kappa,\lambda)$-tree has a
cofinal branch.
\item $\ITP(\kappa,\lambda)$ holds if for every thin $(\kappa,\lambda)$-tree $T$
and every sequence $\langle d_x \mid x \in P_\kappa\lambda \rangle$ such that
for all $x$, $d_x \in T_x$, there is a cofinal branch $b$ through $T$ such that
$\{ x \mid b \upharpoonright x = d_x \}$ is stationary.
\end{enumerate}
\end{definition}

Note that $\TP(\kappa,\kappa)$ is just the tree property at $\kappa$. We say
that $\kappa$ has the strong tree property if $\TP(\kappa,\lambda)$ holds for
all $\lambda \geq \kappa$ and $\kappa$ has the super tree property if
$\ITP(\kappa,\lambda)$ holds for all $\lambda \geq \kappa$.  The notion of
\emph{thinness} was isolated by Weiss \cite{weiss}.  It allowed for the
reformulation of previous results of Jech \cite{jech} and Magidor
\cite{magidorsc} characterizing strong and super compactness respectively. In
particular an inaccessible cardinal $\kappa$ is strongly compact if and only if
it has the strong tree property and it is supercompact if and only if it has the
super tree property.

In order to state our main theorems, we give some standard definitions.  We
start with the square principles $\square_{\mu,\lambda}$, which were first
defined in \cite{Schimmerling95}.

\begin{definition} Let $\lambda \leq \mu$ be cardinals.  We
define a $\square_{\mu,\lambda}$-sequence and say that $\square_{\mu,\lambda}$
holds if and only if there is a $\square_{\mu,\lambda}$-sequence.  A sequence
$\langle \mathcal{C}_\alpha \mid \alpha < \mu^+ \rangle$ is a
$\square_{\mu,\lambda}$-sequence if 
\begin{enumerate}
\item for all $\alpha < \mu^+$, $1 \leq \vert \mathcal{C}_\alpha \vert \leq
\lambda$,
\item for all $\alpha < \mu^+$ and all $C \in \mathcal{C}_\alpha$, $C$ is club
in $\alpha$ and $\otp(C) \leq \mu$, and
\item for all $\alpha < \mu^+$ and all $C \in \mathcal{C}_\alpha$, if $\beta \in
\acc(C)$, then $C \cap \beta \in \mathcal{C}_\beta$.
\end{enumerate}
\end{definition}
Note that $\square_{\mu,\mu}$ is typically written $\square_\mu^*$ and by a
theorem of Jensen \cite{jensen} it is equivalent to the existence of a special
$\mu^+$-Aronszajn tree.  We will sometimes write $\square_{\mu,<\lambda}$ with
the obvious interpretation.

Next we give some definitions from \PCF theory.  Let $\mu$ be a singular cardinal
of cofinality $\omega$.  Let $\langle \mu_i \mid i < \omega \rangle$ be an
increasing sequence of regular cardinals cofinal in $\mu$.  We define an
ordering on $\prod_{i<\omega}\mu_i$ as follows.  Let $f,g \in
\prod_{i<\omega}\mu_i$ and set $f <^* g$ if and only if there is a $j< \omega$
such that for all $i \geq j$, $f(i) < g(i)$.  We say that $\langle f_\beta \mid
\beta < \mu^+ \rangle$ is a \emph{scale of length $\mu^+$ in $\prod_{i<\omega}
\mu_i$} if it is increasing and cofinal in $\prod_{i<\omega} \mu_i$ under the
$<^*$ ordering. A point $\alpha<\mu^+$ of uncountable cofinality is a
\emph{good point} (respectively \emph{very good}) if there is an unbounded
(respectively club) $A \subseteq \alpha$ such that $\langle f_\beta(i) \mid
\beta \in A \rangle$ is strictly increasing for all large $i$.  If $\alpha$ is
not good, then we say that $\alpha$ is a \emph{bad point}.  A scale $\langle
f_\beta \mid \beta < \mu^+ \rangle$ is \emph{good} (respectively \emph{very
good}) if there is a club $F \subseteq \mu^+$ such that each $\alpha \in F$ with
$\cf(\alpha) > \omega$ is good (respectively very good).   \emph{Bad scales} of
length $\mu^+$ are those which are not good.  In particular a bad scale has
stationarily many bad points.

The following theorem of Magidor and Shelah \cite{magidorshelah} shows that
strongly compact cardinals have some accumulated affect on the universe.

\begin{theorem} Let $\lambda$ be a singular limit of strongly compact cardinals.
Then $\lambda^+$ has the tree property.  \end{theorem}

Note that it is consistent that $\kappa$ is supercompact but the tree
property fails for every cardinal above $\kappa$ \cite{bendavidshelah}. This
shows that the fact that $\lambda^+$ has the tree property could not be deduced
only from the existence of a single strongly compact below it, but we had to use
the full power of the cofinal sequence of the strongly compact cardinals.

The following theorem of Shelah shows that supercompact cardinals have an effect
on the \PCF structure above them.

\begin{theorem} If $\kappa$ is supercompact and $\mu > \kappa$ is singular cardinal
of cofinality $\omega$, then all scales of length $\mu^+$ are bad. In particular,
there are no very good scales of length $\mu^+$. \end{theorem}

In this paper, we show that it is not possible to replace the large cardinal
assumptions in the above theorems with the super tree property.

\begin{theorem}
It is consistent relative to $\omega$ many supercompact cardinals that for $2
\leq n <\omega$ the super tree property holds at $\aleph_n$ and
$\square^*_{\aleph_{\omega}}$. Similarly it is consistent that the super tree
property holds at each $\aleph_n$ for $2 \leq n < \omega$ and there is a very
good scale of length $\aleph_{\omega+1}$.
\end{theorem}

In fact we get the consistency of $\square_{\aleph_\omega, < \aleph_{\omega}}$
together with the super tree property at every $\aleph_n$ with $n\geq 2$. Note
that this is the strongest possible square at this scenario, since the super
tree property at $\aleph_n$ implies the failure of $\square_{\lambda, <
\aleph_n}$ for all $\lambda \geq \aleph_n$ by a theorem of Weiss \cite{weiss}.

The theorems of this paper can be seen as extensions of work of the second
author \cite{ungercf} who showed that the super tree property at the
$\aleph_n$'s is consistent with the combinatorial principle $\aleph_{\omega+1}
\in I[\aleph_{\omega+1}]$, which implies that all scales of length
$\aleph_{\omega+1}$ are good and is a weakening of $\square_{\aleph_\omega}^*$.

Throughout the paper we work in \ZFC. Any large cardinals assumption will be specified. Our terminology is mostly standard. We denote by $V[\mathbb{P}]$ the generic extension of the model $V$ by a generic filter for $\mathbb{P}$. We write "$V[\mathbb{P}]\models\phi$" for the assertion "$V\models\Vdash_{\mathbb{P}} \phi$".   
\section{main theorem}

Towards the proof of the main theorem we need the following lemma.

\begin{lemma} \label{indestructible}
Let $\langle \kappa_n \mid n < \omega \rangle$ be an increasing sequence of
supercompact cardinals. There is a forcing extension in which for all
$n<\omega$, $\kappa_n = \aleph_{n+2}$, the super tree property holds at
$\aleph_{n+2}$ and it is indestructible under any $\aleph_{n+2}$-directed closed
forcing.
\end{lemma}

To prove this we repeat the argument from Theorem 7.5 of \cite{ungercf} in the
presence of this extra $\aleph_n$-directed closed forcing.  In particular the
conclusion of the lemma holds in Cummings and Foreman's \cite{cf} model for the
tree property at $\aleph_n$ for all $n \geq 2$.  We will follow the notation of
\cite{ungercf} closely.  The interested reader is advised to have a copy of it
on hand.  The less interested reader can take the lemma as a black box.

\begin{proof} Let $\mathbb{R}_\omega$ be the Cummings-Foreman iteration defined
from the sequence $\langle \kappa_n \mid n < \omega \rangle$.
In the extension by $\mathbb{R}_\omega$, let $\mathbb{X}$ be
$\aleph_{n+2}$-directed closed.  Working in $V[\mathbb{R}_{n+1}]$, we let
$\mathbb{A}_\mathbb{X} =
\mathcal{A}(\mathbb{X},\mathbb{R}_\omega/\mathbb{R}_{n+1})$ be the forcing of
$\mathbb{R}_\omega/\mathbb{R}_{n+1}$-terms for elements of $\mathbb{X}$.
Clearly $\mathbb{A}_\mathbb{X}$ is $\kappa_n$-directed closed in
$V[\mathbb{R}_{n+1}]$.  By increasing the
amount of supercompactness if necessary we can find a generic embedding with
critical point $\kappa_n$ and domain
$V[\mathbb{R}_\omega][\mathbb{A}_\mathbb{X}]$ using the argument from Section 3
of \cite{ungercf}.  We do this by incorporating $\mathbb{A}_\mathbb{X}$ into the
name returned by $j(F)(\kappa_n)$ where $F$ is the Laver function.

We fix a thin $(\aleph_{n+2},\lambda)$-tree $T$ and a sequence $\langle d_x \mid
x  \in P_\kappa\lambda \rangle$ such that for all $x$, $d_x \in T_x$.  Using the
generic embedding, we have a cofinal branch $b:\lambda \to 2$ through $T$ such
that the set $\{x \mid b \upharpoonright x = d_x \}$ is stationary.  By the
analogs of Lemmas 4.1 and 4.2 for our embedding, we have that $b$ is in the
extension of $V[\mathbb{R}_\omega][\mathbb{X}]$ by the product of
$\mathbb{S}_\mathbb{X} =
\mathcal{S}(\mathbb{X},\mathbb{R}_\omega/\mathbb{R}_{n+1})$ (a quotient forcing
defined from $\mathbb{A}_\mathbb{X}$) and the forcing from Lemma 4.2 of
\cite{ungercf}.

It remains to show that this forcing cannot have added the branch.  To do this
we just incorporate $\mathbb{S}_\mathbb{X}$ with the other $\mathbb{S}$ forcings
from Lemma 4.2 of \cite{ungercf}.  In particular we show that $\mathbb{S}_X$ is
$\aleph_{n+1}$-closed and $<\aleph_{n+2}$-distributive over
$M[\mathbb{R}_\omega][\mathbb{X}]$.  The closure is immediate from Lemma 2.12 of
\cite{ungercf} and the fact that $\mathbb{R}_\omega/\mathbb{R}_{n+1}$ is
$<\aleph_{n+1}$-distributive in $V[\mathbb{R}_{n+1}]$.  The distributivity is
immediate from the $\aleph_{n+2}$-directed closure of $\mathbb{A}_\mathbb{X}$ in
$M_n= V[\mathbb{R}_{n+1}]$ and the end of the proof of Lemma 4.4 on
\cite{ungercf}.  This finishes the proof. \end{proof}

Let us define next the forcing notions for adding and threading weak square as well as forcing notions for adding and threading a very good scale.

\begin{definition}
Let $\bS$ be the forcing notion for adding $\square_{\mu, <\mu}$ using bounded approximations. A condition in $\bS$ is a sequence of the form $\langle \mathcal{C}_\alpha \mid \alpha \leq \gamma\rangle$ where:
\begin{enumerate}
\item $\gamma < \mu^+$.
\item $0 < |\mathcal{C}_\alpha| < \mu$ for all limit $\alpha \leq \gamma$.
\item Every $C\in\mathcal{C}_\alpha$ is closed unbounded subset of $\alpha$ with
$\otp(C)\leq\mu$.
\item If $\beta\in\acc C$, $C\in\mathcal{C}_\alpha$ then $C\cap \beta\in\mathcal{C}_\beta$.
\end{enumerate}

We order $\bS$ by end extension.
\end{definition}
The generic filter for $\bS$ is a $\square_{\mu,<\mu}$-sequence. 
For such sequences $\mathcal{C}$, we define the threading forcing $\bT_\rho$. The elements of $\bT$ are members of $\mathcal{C}_\alpha$ for some $\alpha < \mu^+$, with order type $<\rho$, ordered by end extension. 

The following fact is standard:
\begin{claim}
$\bS\ast\bT_\rho$ has a $\rho$-directed closed dense subset.
\end{claim}

It follows that forcing with $\mathbb{S}$ preserves cardinals up to $\mu^+$.
Let us define now a forcing for adding a very good scale at $\mu^+$ and the
corresponding threading forcing.

\begin{definition}
Let $\mu$ be a singular cardinal of countable cofinality, and let $\langle \mu_n
\mid n < \omega\rangle$ be an increasing sequence of regular cardinals cofinal
at $\mu$.

The forcing notion $\bS_{sc}$ is the forcing for adding a very good scale using bounded approximations. A condition in $\bS_{sc}$ is a pair $\langle d, s\rangle$ where:
\begin{enumerate}
\item $s = \langle g_\alpha \mid \alpha \leq \gamma\rangle$, $\gamma < \mu^+$,
where $g_\alpha \in \prod_{n < \omega}\mu_n$, increasing modulo finite. 
\item $d\subseteq \gamma + 1$ is closed set of very good points for $s$.
\end{enumerate}
We order $\mathbb{S}_{sc}$ by end extension.
\end{definition}

It is straightforward to see that this forcing adds a very good scale of length
$\mu^+$. Similarly to the square forcing, there is a natural threading forcing
for $\bS_{sc}$. For $n < \omega$, let $\bT_{sc, n}$ be the forcing notion for
adding a club $E$ in $\mu^+$ such that for every $\alpha < \beta$ in $E$ and
$m\geq n$, $g_\alpha(m) < g_\beta(m)$ with approximations of ordertype at most
$\mu_n$ ordered by end extension.

\begin{lemma}
For $n<\omega$, $\bS_{sc}\ast\bT_{sc, n}$ has a $\mu_n$-directed closed dense
subset.
\end{lemma}
\begin{proof} Let us show first that $\bS_{sc}$ is $<\mu^+$-distributive.
\begin{claim}\label{claim:Sscale strategically closed}
$\bS_{sc}$ is $<\mu$-strategically closed. 
\end{claim}
\begin{proof}
Let us define for an ordinal $\rho < \mu$, a winning strategy for the generic game of length $\rho$. Let us pick $n$ such that $\rho < \mu_n$. Assume that we played the first $\beta$ steps in the game and let $\langle p_\alpha \mid \alpha < \beta\rangle$ be the play so far. Let us denote $p_\alpha = \langle d_\alpha, s_\alpha\rangle$ and let $\gamma_\alpha = \max \dom s_\alpha$, $\langle f_i \mid i < \sup_\alpha \gamma_\alpha\rangle = \bigcup_{\alpha < \beta} s_\alpha$, the scale constructed so far. Let $d = \bigcup_{\alpha < \beta} d_\alpha$. 

The strategy will be to pick $p_\beta = \langle e, \langle f_i \mid i < \sup \gamma_\alpha\rangle^{\smallfrown} \langle g\rangle \rangle$ where:
\begin{enumerate}
\item $e = d$ if $\beta$ is not a limit ordinal and otherwise $e = d\cup\{\gamma_\beta\}$, $\gamma_\beta = \sup_{\alpha < \beta} \gamma_\alpha$.
\item  $g$ is an upper bound (modulo finite errors) of all $f_i$, $i < \sup \gamma_\alpha$ and for all $\alpha < \beta$, and $m \geq n$, $g(m)\geq f_{\gamma_\alpha}(m)$.   
\end{enumerate}
We need to verify that $\gamma_\beta$ is indeed a good point whenever $\beta$ is a limit ordinal. The club $\{\gamma_\alpha \mid \alpha < \beta\}$ witnesses that this is the case.
\end{proof}
Since $\mu$ is singular, we conclude that $\bS_{sc}$ is $\mu^+$-distributive. Therefore the elements of $\bT_{sc, n}$ are members of the ground model. Let us show now that $\bS_{sc}\ast\bT_{sc, n}$ contains a dense $\mu_n$-directed closed subset. 

Let $D$ be the set of all $\langle \langle d, s\rangle, \check{c}\rangle \in \bS_{sc}\ast\bT_{sc,n}$ such that $c\in V$, and $\max \dom s = \max d = \max c$. $D$ is $\mu_n$-directed closed, since for every sequence of pairwise compatible elements of length $\rho<\mu_n$, $\{ \langle \langle s_i, d_i\rangle, \check{c_i}\rangle \mid i < \rho \}$, has a lower bound. The only thing that we need to verify is that one can add a member of $\prod_{n<\omega} \mu_n$ in the top of the scale $\bigcup_{i < \rho} s_i$ in a way that will make it a good point and this is witnessed by the club $\bigcup_{i<\rho} c_i$.

Let us show that $D$ is dense. Let $p\in\bS_{sc}\ast\bT_{sc, n}$. By extending $p$, if necessary, we may assume that $p = \langle \langle s, d\rangle, \check{c}\rangle$. Using the strategy, we know that we can extend $\langle s, d\rangle$ to a condition $\langle s^\prime, d^\prime\rangle$ such that $\max \dom s^\prime = \max d$. Moreover, we may pick the last element in $s^\prime$ to be above all elements in $s\restriction c$ in all its coordinates, besides the first $n$. Thus, we can extend $c$ to include $\max d$.
\end{proof}
The next two lemmas show that the threading forcing corresponding to the weak square forcing and the very good scale forcing cannot add a new branch to a $P_\kappa \lambda$-tree.

\begin{lemma} \label{pres1} Let $\kappa, \lambda$ be regular cardinals such
that $\kappa$ is not strong limit and $\kappa \leq \lambda$. Let $\nu \leq \mu$
be cardinals with $\kappa < \nu$. Let $\mathbb{S}$ be the forcing for adding a
$\square_{\mu,<\nu}$ sequence. Let $\mathbb{T}$ be the threading forcing with
approximations of order type less than $\kappa$. Then forcing with $\mathbb{T}$
over $V[\mathbb{S}]$ does not add any new branch to a thin $P_\kappa \lambda$
tree.  \end{lemma}

\begin{proof}
Assume otherwise, and let $\dot{b}$ be a name for this branch. 
Let $\rho$ be the least cardinal for which $2^\rho \geq \kappa$. Since $\kappa$ is not 
strong limit, $\rho<\kappa$ and $2^{<\rho} < \kappa$. 
Let $s\in\mathbb{S}$ and $t\in\mathbb{T}$ be arbitrary. 
Since $\dot{b}$ is new, one can extend the condition 
$\langle s, t\rangle \in \mathbb{S}\ast\mathbb{T}$ to pair of
conditions $\langle s^\prime, t_0\rangle$, $\langle s^\prime, t_1 \rangle$
that force different values for $\dot{b}$ at some $x\in P_\kappa \lambda$.

Let $\{\eta_i \mid i < 2^{<\rho}\}$ be an enumeration of all elements of $^{<\rho}2$, such that if $\eta_i \trianglelefteq \eta_j$ then $i \leq j$.
Let us define by induction a sequence of conditions $s_\alpha$, $\alpha < 2^{<\rho}$ and $t_{\eta}$, $\eta\in \,^{<\rho}2$, such that:
\begin{enumerate}
\item For every $\alpha < 2^{<\rho}$, $s_\alpha \Vdash t_{\eta_\alpha}\in\bT$.
\item For every $\eta\in \,^{<\rho}2$ there is $x_\eta\in P_\kappa \lambda$ such that $s_\alpha$ forces $t_{\eta^\smallfrown\langle 0\rangle}, t_{\eta^\smallfrown\langle 1\rangle}$ are stronger than $t_\eta$ and force different values for $\dot b$ in $x_\eta$ where $\alpha$ is such that $\eta_\alpha =\eta$. 
\item If $\eta_\alpha \triangleleft \eta$ and $\eta_\alpha\neq\eta$ then $\max t_\eta \geq \max \dom s_\alpha$.   
\end{enumerate}

For $\eta\in\,^\rho 2$ let $t_\eta = \bigcup_{i<\rho} t_{\eta\restriction i}$. Let $\mathcal{A}$ be a set of $\kappa$ many different elements in $2^{\rho}$.

Let $s_\star$ be the condition $\bigcup_\alpha s_\alpha ^\smallfrown \langle \{t_\eta \mid \eta\in\mathcal{A}\}\rangle$. Let $x_\star = \bigcup_{\eta\in\,^{<\rho} 2} x_{\eta}$. $s_\star\Vdash t_{\eta}\in\bT$ for all $\eta\in\mathcal{A}$, and it forces that if $\eta\neq\eta^\prime$ and $\eta,\eta^\prime\in\mathcal{A}$ then $t_{\eta}, t_{\eta^\prime}$ force different values on $\dot{b}$ in $x_\star$ - but this contradicts the assumption that the tree is thin.   
\end{proof}

\begin{lemma} \label{pres2}
Let $\mu$ be a singular cardinal of countable cofinality, $\lambda_n \to \mu$. Force with $\bS_{sc}$. If $\kappa < \lambda_n$ is not a strong limit, then the forcing $\bT_{sc, n}$ cannot add a new branch to a $P_\kappa \lambda$ tree.
\end{lemma}

The proof is essentially the same as the previous lemma.  We are now ready to
complete the proof of the main theorem.

\begin{proof} Let $W$ be the Cummings-Foreman model for the tree property at
$\aleph_n$ for $n \geq 2$.  Let $\mathbb{S}$ be the forcing to add either a
$\square_{\aleph_\omega,<\aleph_\omega}$-sequence or a very good scale of
length $\aleph_{\omega+1}$.  We claim that $W[\mathbb{S}]$ is the desired
model.  Clearly $W[\mathbb{S}]$ has either the appropriate weak square sequence
or a very good scale based on the choice of $\mathbb{S}$.  So it remains to
show that the super tree property holds at $\aleph_{n+2}$ for all $n<\omega$.

Let $n<\omega$ and let $\mathbb{T}_n$ be the appropriate threading forcing so
that $\mathbb{S}*\mathbb{T}_n$ is $\aleph_{n+2}$-directed closed in $W$.  By
Lemma \ref{indestructible}, the super tree property holds at $\aleph_{n+2}$ in
$W[\mathbb{S}*\mathbb{T}_n]$.  So by either Lemma \ref{pres1} or \ref{pres2}
applied with $\kappa = \aleph_{n+2}$, the super tree property holds at
$\aleph_{n+2}$ in $W[\mathbb{S}]$.  \end{proof}

\providecommand{\bysame}{\leavevmode\hbox to3em{\hrulefill}\thinspace}
\providecommand{\MR}{\relax\ifhmode\unskip\space\fi MR }
\providecommand{\MRhref}[2]{%
  \href{http://www.ams.org/mathscinet-getitem?mr=#1}{#2}
}
\providecommand{\href}[2]{#2}

\end{document}